\documentclass[a4paper, 12pt, oneside]{amsart}

\usepackage[top=35mm,bottom=15mm,left=20mm,right=20mm]{geometry}
\usepackage{mathrsfs}
\usepackage{amsfonts}
\usepackage{amssymb}
\usepackage{amsxtra}
\usepackage{color}
\usepackage{dsfont}
\usepackage[colorinlistoftodos,  textsize=tiny]{todonotes} 					 %% color comments in pdf
\presetkeys{todonotes}{fancyline, color=blue!30}{}

\usepackage[english,polish]{babel}
\usepackage[colorlinks,citecolor=blue,urlcolor=blue,bookmarks=true]{hyperref}
\hypersetup{
pdfpagemode=UseNone,
pdfstartview=FitH,
pdfdisplaydoctitle=true,
pdfborder={0 0 0}, % No link borders
pdftitle={Heat content for convolution semigroups},
pdfauthor={Wojciech Cygan, Tomasz Grzywny},
pdflang=en-US
}

\newcommand{\sprod}[2]{\langle {#1}, {#2}\rangle}
\newcommand{\norm}[1]{{\lVert #1 \rVert}}

\newcommand{\RR}{\mathbb{R}}
\newcommand{\PP}{\mathbb{P}}
\newcommand{\EE}{\mathbb{E}}

\newcommand{\IndFun}{\textbf{1}_}
\newcommand{\Per}{\mathrm{Per}}
\newcommand{\calR}{\mathcal{R}}

\newcommand{\ud}{\mathrm{d}}
\newcommand{\psiI}{\psi^{-}(1/t)}
\newcommand{\g}{g_\Omega}
\newcommand{\RInf}{\calR_{\alpha} }

\newcommand{\pl}[1]{\foreignlanguage{polish}{#1}}

\newtheorem{theorem}{Theorem}
\newtheorem{proposition}{Proposition}
\newtheorem{lemma}{Lemma}

\theoremstyle{definition}
\newtheorem{example}{Example}

\title{Heat content for convolution semigroups}

\author{Wojciech Cygan}
\thanks{W.~Cygan was supported by National Science Centre (Poland): grant DEC-2013/11/N/ST1/03605 and by Austrian Science Fund project FWF P24028.}

\author{Tomasz Grzywny}
 \thanks{T.~Grzywny was partially supported by  National Science Centre (Poland): grant 2015/17/B/ST1/01043.}

\address{
		Wojciech Cygan\\
		Instytut Matematyczny\\
		Uniwersytet \pl{Wroc{\lll}awski}\newline
		Pl. Grun\-waldzki 2/4\\
		50-384 \pl{Wroc{\lll}aw}\\
		Poland}
\email{wojciech.cygan@uwr.edu.pl}

\address{
	Tomasz Grzywny\\
	 \pl{Wydzia{\lll}} Matematyki\\
	Politechnika \pl{Wroc{\lll}awska}\newline
	Wyb. \pl{Wyspia\'{n}skiego} 27
	50-370 \pl{Wroc\l{}aw}\\
	Poland}
\email{tomasz.grzywny@pwr.edu.pl}


\begin{document}
\selectlanguage{english}

\begin{abstract}
Let $\mathbf{X}=\{X_t\}_{t\geq 0}$ be a L\'{e}vy process in $\RR^d$ and $\Omega$ be an open subset of $\RR^d$ with finite Lebesgue measure. In this article we consider the quantity 
$H (t) = \int_{\Omega}\PP_{x} (X_t\in \Omega ^c)\, \ud x$ 
which is called the heat content.
We study its asymptotic behaviour as $t$ goes to zero for isotropic L\'{e}vy processes under some mild assumptions on the characteristic exponent. We also treat the class of L\'{e}vy processes with finite variation in full generality.
\end{abstract}

\maketitle
\section{Introduction}

Let $\mathbf{X}=\{X_t\}_{t\geq 0}$ be a L\'{e}vy process in $\RR ^d$ with the distribution $\PP$  such that $X_0=0$. 
We denote by $p_t(\ud x)$ the distribution of the random variable $X_t$ and 
we use the standard notation $\PP_x$ for the distribution related to the process $\mathbf{X}$ started at $x\in \RR^d$. The characteristic exponent $\psi (x)$, $x\in \RR ^d$, of the process $\mathbf{X}$ is given by the formula
\begin{align}\label{charact_expo}
\psi(x) = \sprod{x}{Ax} - i\sprod{x}{\gamma } - 
\int_{\RR ^d}\left(e^{i\sprod{x}{y}} - 1 - i\sprod{x}{y} \IndFun{\{\norm{y} \leq 1\}} \right) \nu(\ud y) ,
\end{align}
where $A$ is a symmetric non-negative definite $d\times d$ matrix, $\gamma\in \RR ^d$ and $\nu$ is a L\'{e}vy measure, that is
\begin{align}\label{Levy_measure}
\nu(\{0\})=0\quad \mathrm{and}\quad \int_{\RR ^d} \left( 1\wedge \norm{y}^2 \right)\, \nu(\ud y) <\infty .
\end{align}
%% HEAT CONTENT

Let $\Omega$ and $\Omega _0$ be two non-empty subsets of $\RR ^d$ such that $\Omega$ is open and its Lebesgue measure $|\Omega|$  is finite. 
We consider the following quantity associated with the process $\mathbf{X}$,
\begin{align*}
H _{\Omega , \Omega _0} (t) = \int_{\Omega}\PP_{x} (X_t\in \Omega _0)\, \ud x =  \int_{\Omega}\int_{\Omega_0-x}p_t( \ud y)\ud x
\end{align*} 
and we use the notation
\begin{align}\label{heat_cont-H}
H_\Omega (t)= H_{\Omega , \Omega } (t)\quad \mathrm{and}\quad H(t)=H _{\Omega , \Omega^{c}} (t). 
\end{align}
The main goal of the present article  is to study the asymptotic behaviour of $H_\Omega (t)$ as $t$ goes to zero. 
We observe that
\begin{align*}
H_\Omega (t) = |\Omega| - H(t),
\end{align*}
and thus it suffices to work with the function $H(t)$. The function $u(t,x) = \int_{\Omega-x}p_t(\ud y)$ is the weak solution of the initial value problem
\begin{align*}
\frac{\partial}{\partial t}u(t,x) &= -\mathcal{L}\, u(t,x),\quad t>0,\, x\in \RR^d,\\
u(0,x) &= \IndFun{\Omega}(x),
\end{align*}
where $\mathcal{L}$ is the infinitesimal generator of the process $\mathbf{X}$, see \cite[Section 31]{Sato}. Therefore, $H_\Omega (t)$ can be interpreted as the amount of \textit{heat} in $\Omega$ if its initial temperature is one whereas the initial temperature of $\Omega^c$ is zero. In paper \cite{vanDenBerg1_POT}, the author calls the quantity $H_\Omega (t)$ \textit{heat content} and we will use the same terminology. There are a lot of articles where bounds and asymptotic behaviour of the heat content related to Brownian motion, either on $\RR^d$ or on compact manifolds, were studied, see \cite{vanDenBerg1_POT}, \cite{vanDenBerg1}, \cite{vanDenBerg2}, \cite{vanDenBerg3}, \cite{vanDenBerg_4}, \cite{vanDenBerg5}. Recently Acu\~{n}a Valverde \cite{Valverde1} investigated the heat content for isotropic stable processes in $\RR^d$, see also \cite{Valverde2} and \cite{Valverde3}. In this paper we study the small time behaviour of the heat content associated with rather general L\'{e}vy processes in $\RR^d$.

Before we state our results we recall the notion of perimeter. Following \cite[Section 3.3]{Ambrosio_2000}, for any measurable set\footnote{All sets in the paper are assumed to be Lebesgue measurable.} $\Omega \subset \RR ^d$ we define its perimeter $\Per(\Omega)$ as 
\begin{align}\label{Perimeter_def}
\Per(\Omega) = \sup \left\{ \int_{\RR^d}\IndFun{\Omega}(x)\mathrm{div}\, \phi (x)\, \ud x:\, \phi \in C_c^1(\RR^d,\RR^d),\, \norm{\phi}_{\infty}\leq 1 \right\}.
\end{align}
We say that $\Omega$ is of finite perimeter if $\Per(\Omega)<\infty$. It was shown in \cite{Miranda1, Miranda2, Preunkert} that if $\Omega$ is an open set in $\RR^d$ with finite Lebesgue measure and of finite perimeter then
\begin{align*}
\Per(\Omega) = \pi^{1/2}\lim_{t\to 0} t^{-1/2}\int_{\Omega}\int_{\Omega^c}p_t^{(2)}(x,y)\, \ud y\, \ud x,
\end{align*}
where
\begin{align*}
p_t^{(2)}(x,y) = (4\pi t)^{-d/2}e^{-\norm{x-y}^2/4t}
\end{align*}
is the transition density of the Brownian motion $B_{t}$ in $\RR^d$. We also notice that for a non-empty and open set $\Omega$, $\Per(\Omega)>0$.

Recall that for the L\'{e}vy process $\mathbf{X}$ with the transition probability $p_t(\ud x)$ and the L\'{e}vy measure $\nu$ we have
\begin{align*}
\lim_{t\to 0} t^{-1} p_t(\ud x) = \nu (\ud x),\quad \text{vaguely on }  \RR ^d\setminus\{0\}.
\end{align*}
Therefore, we introduce the perimeter $\Per_{\mathbf{X}}(\Omega)$ related to the process $\mathbf{X}$ setting
\begin{align}\label{X_perimeter}
\Per_{\mathbf{X}}(\Omega)= \int_{\Omega}\int_{\Omega ^c-x}\nu (\ud y)\, \ud x .
\end{align}
For instance, if $\mathbf{X}$ is the isotropic (rotationally invariant) $\alpha$-stable process, denoted by $S^{(\alpha)}=(S^{(\alpha)}_t)_{t\geq 0}$, we obtain the well-known $\alpha$-perimeter, which for $0<\alpha <1$ is given by
\begin{align*}
\Per_{S^{(\alpha)}}(\Omega) = \int_{\Omega}\int_{\Omega ^c} \frac{\ud y\, \ud x}{\norm{x-y}^{d+\alpha}}.
\end{align*}
It was proved in \cite{Fusco} that if $\Omega$ has finite Lebesgue measure and is of finite perimeter then $\Per_{S^{(\alpha)}}(\Omega)$ is finite. 
In the present paper we prove (see Lemma \ref{Lemma_Per_X}) that for any L\'{e}vy process with finite variation, cf. \cite[Section 21]{Sato}, and for $\Omega$ of finite measure and of finite perimeter $\Per(\Omega) <\infty$ the quantity $\Per_{\mathbf{X}}(\Omega)$ is finite as well.

After Pruitt \cite{Pruitt}, we consider the following function related to the L\'{e}vy process $\mathbf{X}$. For any $r>0$,
\begin{align}
\begin{aligned}\label{Pruitt_Function}
h(r)= \norm{A}r^{-2} &+ 
 r^{-1} \Big\lvert \gamma +\int_{\RR ^d} y \left(\textbf{1}_{\norm{y} < r}-\textbf{1}_{\norm{y} < 1}\right){}& \nu(\ud y)\Big\rvert  \\
& +\int_{\RR ^d} \left( 1\wedge \norm{y}^2r^{-2}\right) \, \nu(\ud y),
\end{aligned}
\end{align}
where $(A,\gamma ,\nu)$ is the triplet from \eqref{charact_expo} and $\Vert A\Vert = \max_{\Vert x\Vert =1} \norm{Ax}$.

Our first result gives a general upper bound for the heat content related to any L\'{e}vy process in $\RR^d$. 
\begin{theorem}\label{Thm_d>1}
Let $\Omega \subset \RR^d$ be an open set of finite measure $|\Omega| $ and of finite perimeter $\Per (\Omega)$, and set $R=2|\Omega|/\Per(\Omega)$. Let $\mathbf{X}$ be a L\'{e}vy process in $\RR^d$.
Then there is a constant $C_1=C_1(d) >0$ which does not depend on the set $\Omega$ such that 
\begin{align*}
H(t) &\leq C_1\, t\, \Per (\Omega) \int_{\frac{R}{2} \wedge h^{-1}(1/t)}^{R} h(r)\, \ud r,\quad t>0.
\end{align*}
\end{theorem}
%%%%%%%%%%%%%%%%%%%%%%
In Proposition \ref{H_lower_bound} we also prove a similar lower bound for a class of isotropic L\'{e}vy processes with characteristic exponent satisfying the so-called upper scaling condition, see \cite{BGR}.
Let us recall that a L\'{e}vy process $\mathbf{X}$ is isotropic if the measure 
$p_t(\ud x)$ is radial (rotationally invariant) for each $t > 0$, equivalently to the matrix $A=\eta I$ for some $\eta\geq0$, the L\'{e}vy measure $\nu$ is rotationally invariant and $\gamma =0$.

In the next theorem we present the asymptotic behaviour of the heat content under the assumption that the L\'{e}vy process $\mathbf{X}$ is isotropic and its characteristic exponent is a regularly varying function at infinity with index greater than one.
We say that a function $f(r)$ is regularly varying of index $\alpha$ at infinity, denoted by $f\in \RInf$, if for any $\lambda >0$,
\begin{align*}
\lim_{r\to \infty}\frac{f(\lambda r)}{f(r)} = \lambda ^\alpha .
\end{align*}
\begin{theorem}\label{Thm_alpha>1}
Let $\Omega \subset \RR^d$ be an open set of finite measure $|\Omega| $ and finite perimeter $\Per (\Omega)$. If $\mathbf{X}$ is an isotropic L\'{e}vy process in $\RR^d$ with the characteristic exponent $\psi$ such that $\psi \in \RInf$, for some $\alpha \in (1,2]$, then\footnote{Here $\psi ^-=(\psi^*)^-$ is the generalized left inverse of the non-decreasing function $\psi^*(u) = \sup_{s \in [0, u]} \psi(s)$, see Subsection \ref{sec_Levy}.}
\begin{align*}
\lim_{t\to 0} \psiI H(t) =  \pi^{-1}\Gamma(1-1/\alpha) \Per (\Omega) .
\end{align*}
\end{theorem}
%%%%%%%%%%%%%%%%%%%%%%%%
%%%% Theorem fo X of bounded variation

The following theorem deals with L\'{e}vy processes with finite variation. Recall that according to \cite[Theorem 21.9]{Sato} a L\'{e}vy process $\mathbf{X}$ has finite variation on any interval $(0,t)$ if and only if 
\begin{align}\label{Levy_bdd_var_cond}
A=0\quad \mathrm{and}\quad \int_{\norm{x}\leq 1}\norm{x}\nu (\ud x)<\infty .
\end{align}
In this case the characteristic exponent has the following simple form
\begin{align*}
\psi (x) = i\sprod{x}{\gamma _0} + \int_{\RR^d}\left( 1-e^{i\sprod{x}{y}}\right)\nu (\ud y),
\end{align*}
where 
\begin{align}\label{gamma_0}
\gamma _0 = \int_{\norm{y}\leq 1}y\, \nu (\ud y) - \gamma .
\end{align}
We notice that for symmetric L\'{e}vy processes with finite variation we have $\int_{\norm{y}\leq 1}y\, \nu (\ud y) =0$. Thus, for any symmetric L\'{e}vy process with finite variation we have $\gamma _0 = 0$. 
Moreover, for such processes the related function $h$ defined at \eqref{Pruitt_Function} is Lebesgue integrable on every bounded interval. 
As we mentioned before, in front of Lemma \ref{Lemma_Per_X} the quantity $\Per_{\mathbf{X}}(\Omega)$ is finite in the following theorem. For the definition of a directional derivative we refer the reader to Subsection \ref{sec_geom}.
\begin{theorem}\label{Thm_X_bdd_variation}
Let $\mathbf{X}$ be a L\'{e}vy process in $\RR^d$ with finite variation. Let $\Omega \subset \RR^d$ be an open set of finite measure $|\Omega| $ and finite perimeter $\Per (\Omega)$.  Then
\begin{align*}
\lim_{t\to 0}t^{-1}H(t) = \Per_{\mathbf{X}}(\Omega) + \frac{\norm{\gamma_0}}{2} V_{\frac{\gamma_0}{\norm{\gamma_0}}}(\Omega) \IndFun{\RR^d\setminus \{0\}}(\gamma _0),
\end{align*}
where $V_u(\Omega)$ is the directional derivative of the indicator function $\IndFun{\Omega}$ in the direction $u$ on the unit sphere in $\RR^d$.
\end{theorem}

The rest of the paper is organized as follows. We start with a paragraph which gives the list of examples. In Section \ref{sec_Prelim} we present all the necessary facts and tools that we use in the proofs. Section \ref{sec_Proofs} is devoted to the proofs of the above theorems.

\subsubsection*{Notation}
We write $a\wedge b$ for $\min \{a,b\}$ and $a\vee b=\max\{a,b\}$. Positive constants are denoted by $C_1, C_2$ etc. If additionally $C$ depends on some $M$, we write $C=C(M)$. 
We use the notation $f(x) = O(g(x))$ if there is a constant $C>0$ such that $f(x)\leq C g(x)$; $f(x)\asymp g(x)$ if $f(x)=O(g(x))$ and $g(x)=O(f(x))$; $f(x)\sim g(x)$ at $x_0$ if $\lim_{x\to x_0}f(x)/g(x)=1$.
$\mathbb{S}^{d-1}$ stands for the unit sphere in $\RR^d$ and $\sigma (\ud u) = \sigma ^{d-1}(\ud u)$ is the surface measure.
\vspace*{0,2cm}

\subsection{Examples}
First we consider the isotropic (rotationally invariant) $\alpha$-stable process in $\RR^d$. The  following example shows that our theorems can be regarded as extensions of the results contained in papers \cite{Valverde1} and \cite{Valverde2}. 
\begin{example}
Let $S^{(\alpha)} = (S_t^{(\alpha)})_{t\geq 0}$ be the isotropic $\alpha$-stable process in $\RR^d$ with $\alpha \in (0,2)$.  We recall that the characteristic exponent of $S^{(\alpha )}$ is $x\mapsto c\norm{x}^{\alpha}$, for some $c>0$, see \cite[Theorem 14.14]{Sato}. The L\'{e}vy measure $\nu$ of $S^{(\alpha)}$ has the form
\begin{align*}
\nu (\ud x) = \frac{c_1\, \ud x}{\norm{x}^{d+\alpha}},\quad \mathrm{for\ some}\ c_1>0.
\end{align*}
The related function $h$ defined in \eqref{Pruitt_Function} turns into $h(r) = c_2/r^{\alpha}$, for some $c_2>0$. Let $\Omega \subset \RR^d$ be an open set of finite measure $|\Omega| $ and finite perimeter $\Per (\Omega)$ and let $R=2|\Omega|/\Per(\Omega)$. Then, by Theorem \ref{Thm_d>1}, for any $\alpha \in (0,2)$, 
\begin{align*}
H(t)\leq C_1 \Per(\Omega)\, t\int_{\frac{R}{2} \wedge t^{1/\alpha}}^{R}r^{-\alpha}\, \ud r ,\ \mathrm{for\ all}\ t>0,
\end{align*}
and, by Proposition \ref{H_lower_bound}, for $\alpha\in[1,2)$ and $t$ small enough,
\begin{align*}
H(t)\geq C_2\Per(\Omega)\, t\int_{t^{1/\alpha}}^{R}r^{-\alpha}\, \ud r .
\end{align*}
In particular, for $\alpha = 1$ we get
\begin{align*}
\limsup_{t\to 0}\frac{H(t)}{t\log (1/t)}\leq C_1\Per(\Omega)\quad \mathrm{and}\quad \liminf_{t\to 0}\frac{H(t)}{t\log(1/t)}\geq C_2 \Per(\Omega).
\end{align*}
For $\alpha \in (1,2)$, by Theorem \ref{Thm_alpha>1},
\begin{align*}
\lim_{t\to 0} t^{-1/\alpha}H(t) = \pi^{-1}\Gamma(1-1/\alpha) \Per (\Omega)
\end{align*}
and for $\alpha \in (0,1)$, by Theorem \ref{Thm_X_bdd_variation},
\begin{align*}
\lim_{t\to 0}t^{-1}H(t) = \Per_{S^{(\alpha)}}(\Omega).
\end{align*}
Here $\gamma _0 = 0$ according to the comments following equation \eqref{gamma_0}.
\end{example}
%%%%%%%%%%%%%%%%%%%%
\begin{example}
Let $\mathbf{X}$ be a pure jump (i.e. $A=0$ and $\gamma=0$ in \eqref{charact_expo}) isotropic L\'{e}vy process in $\RR^d$ such that its L\'{e}vy measure $\nu$ has  the form
\begin{align}\label{Levy_meas_reg}
\nu (dx ) = ||x||^{-d}g(1/||x||)dx,\quad \mathrm{for\ some}\ g \in \mathcal{R}_{\alpha},\ \alpha \in (0,2).
\end{align}
By \cite[Proposition 5.1]{CGT}, we conclude that $\psi \in \RInf$. Hence for such processes, for $1<\alpha <2$,
\begin{align*}
\lim_{t\to 0} \psiI H(t) =  \pi^{-1}\Gamma(1-1/\alpha) \Per (\Omega) ,
\end{align*}
and, for $0<\alpha <1$,
\begin{align*}
\lim_{t\to 0}t^{-1}H(t) = \Per_{\mathbf{X}}(\Omega).% + \norm{\gamma_0}%V_{\frac{\gamma_0}{\norm{\gamma_0}}}(\Omega)/2.
\end{align*}
Typical examples of isotropic L\'{e}vy processes satisfying \eqref{Levy_meas_reg} are
\begin{enumerate}
		\item truncated stable process: $g(r)=r^{-\alpha} \IndFun{(0, 1)}(r)$;
		\item tempered stable process: $g(r)=r^{-\alpha} e^{-r}$;
		\item isotropic Lamperti stable process: $g(r)=re^{\delta r}(e^r-1)^{-\alpha-1}$, $\delta<\alpha+1$;
		\item layered stable process: $g(r)=r^{-\alpha} \IndFun{(0,1)(r)} + r^{-\alpha_1} \IndFun{[1, \infty)}(r),\ \alpha_1\in(0,2)$.
	\end{enumerate}
\end{example}
%%%%%%%%%%%%%%
\begin{example}
Let $\mathbf{X}$ be a L\'{e}vy process which is the independent sum of the Brownian motion and the isotropic $\alpha$-stable process in $\RR^d$. Then $\mathbf{X}$ is isotropic and its characteristic exponent is
 $\psi (x) = \eta ||x||^2+c||x||^\alpha$, for some $\eta ,c >0$ and $\alpha \in (0,2)$. Clearly we have $\psi \in \mathcal{R}_2$ and whence, by Theorem \ref{Thm_alpha>1},
 \begin{align*}
\lim_{t\to 0} t^{-1/2} H(t) =  \sqrt{\frac{\eta}{\pi}} \Per (\Omega) . 
 \end{align*}
\end{example}
%%%%%%%%%%%%%%%55
\begin{example}
Take $\alpha \in (0,2)$ and let $\mathbf{X}$ be a symmetric L\'{e}vy process in $\RR$ which is the independent sum of the isotropic $\alpha $-stable process and a L\'{e}vy process of which the L\'{e}vy measure $\nu$ has the form
\begin{align}\label{leevy_m}
\nu(\ud x)=\sum_{k=1}^{\infty} 2^{k\alpha /2}\left(\delta_{2^{-k}}( \ud x)+\delta_{-2^{-k}}(\ud x)\right),
\end{align}
where $\delta _x$ stands for the Diraac measure at $x$.
According to Subsection \ref{sec_Levy}, the characteristic exponent $f(x)$ of the process related to $\nu(\ud x)$ has the form
\begin{align*}
f(x) = 2\int_0^\infty \left(1-\cos(xu)\right)\nu (\ud u).
\end{align*}
The characteristic exponent of the isotropic $\alpha$-stable process is $x\mapsto c|x|^\alpha$, for some $c>0$, see \cite[Theorem 14.14]{Sato}, and whence, by independence,
the characteristic exponent of $\mathbf{X}$ equals to
\begin{align*}
\psi (x) = c|x|^\alpha + f(x),\quad c>0.
\end{align*}
Since $1-\cos (v)\asymp v^2$, for $0<v<1$, we have for $x>0$,
\begin{align}\label{f_estimate}
 C x^2\int^{1/x}_0 u^2\nu(\ud u) \leq f(x)\leq 4\nu \left((1/x,\infty)\right) + x^2\int^{1/x}_0 u^2\nu(\ud u),
\end{align}
where $\nu \left((1/x,\infty)\right)$ is the $\nu$-measure of the half-line $(1/x,\infty )$.
Using formula \eqref{leevy_m} we obtain that for $x\geq 1$,
\begin{align*}
\int^{1/x}_0 u^2\nu(\ud u) = \sum_{k\geq \log_2 x}\!\! 2^{(\alpha /2-2)k} \asymp 2^{(\alpha /2-2)\log_2 x} = x^{\alpha /2 -2}.
\end{align*}
%where $\log _2x$ is the logarithm to the base $2$.
Similarly we have, for $x> 2$,
\begin{align*}
\nu \left( (1/x,\infty) \right) = \sum _{1\leq k<\log_2 x}2^{\alpha k/2}\leq \sum _{1\leq k\leq [\log_2 x]}2^{\alpha k/2}
=\frac{1-2^{\alpha ([\log_2x ]+1)/2}}{1-2^{\alpha /2}}
\asymp x^{\alpha /2},
\end{align*}
where $[x]$ stands for the integer part of $x$. Hence, by \eqref{f_estimate}, $f(x)\asymp |x|^{\alpha /2}$, for $|x|> 2$. 
We obtain that $\psi (x) \sim c|x|^\alpha$ at infinity and thus, for $\alpha >1$,
\begin{align*}
\lim_{t\to 0}t^{-1/\alpha}H(t) = c^{1/\alpha} \pi^{-1}\Gamma(1-1/\alpha) \Per (\Omega) 
\end{align*}
and, for $\alpha <1$,
\begin{align*}
\lim_{t\to 0}t^{-1}H(t) = \Per_{\mathbf{X}}(\Omega).% + \norm{\gamma_0}%V_{\frac{\gamma_0}{\norm{\gamma_0}}}(\Omega)/2.
\end{align*}
\end{example}

The next example shows that in the case when $\mathbf{X}$ is not isotropic then the constant in Theorem \ref{Thm_alpha>1} may depend on the process.
\begin{example}
For $\alpha>1$ and $\ell \in \mathcal{R}_0$
we consider a L\'{e}vy process $\mathbf{X}$ in $\RR$ with the L\'{e}vy measure $\nu$ of the form
\begin{align*}
\nu(\ud x)=\left( c_1f(1/x)x^{-1}\IndFun{\{x>0\}} + c_2f(1/|x|)|x|^{-1}\IndFun{\{x<0\}}\right)\ud x,
\end{align*}
where
$f(r)=r^\alpha\ell(r)$, for $r\geq 1$ and $f(r)=r^\alpha$ for  $r<1$ and
for some constants $c_1,c_2\geq  0$ such that $c_1+c_2>0$. The corresponding characteristic exponent we call $\psi$.

Let $S$ be the non-symmetric $\alpha$-stable distribution in $\RR$ with the L\'{e}vy measure given by 
$\left( c_1x^{-1-\alpha}\IndFun{\{x>0\}} + c_2|x|^{-1-\alpha}\IndFun{\{x<0\}}\right)\ud x$ and with the characteristic exponent $\psi ^{(\alpha)}$.

We observe that $f^{-1}(1/t)X_t$ converges in law to $S$. Indeed, it is enough to prove the convergence of characteristic functions and this holds since we easily get that for any $x$,
\begin{align*}
\lim_{t\to 0}t\psi(xf^{-1}(1/t))=\psi^{(\alpha)}(x).
\end{align*}
For $\Omega=(a,b)$ we have 
\begin{align*}
H(t) = \int_a^b \PP (X_t\leq a-x)\, \ud x + \int_a^b \PP (X_t\geq b-x)\, \ud x.
\end{align*}
A suitable change of variable in both integrals yields
\begin{align*}
H(t) = \int_0^{R} \PP (|X_t|\geq x)\, \ud x .
\end{align*}
Hence,
\begin{align*}
\lim_{t\to 0}f^{-1}(1/t)H(t)=\lim_{t\to 0} \int^{(b-a)f^{-1}(1/t)}_0\!\!\!\! \PP(f^{-1}(1/t)|X_t|>u)\ud u = \int^\infty_0\PP(|S|>u)\ud u =\EE|S|.
\end{align*}
\end{example}

%%%%%%%%%%%%%%%%%%%%%%%%%%%%%
\section{Preliminaries}\label{sec_Prelim}
In this section we collect all the necessary objects and facts that we use in the course of our study. We start with the short presentation of the geometrical tools.
\subsection{Geometrical issues}\label{sec_geom}
We refer the reader to \cite{Ambrosio_2000} for a detailed discussion on functions of bounded variation and related topics.

Let $G\subseteq \RR ^d$ be an  open set and $f:G\rightarrow \RR$, $f\in L^{1}(G)$. The total variation of $f$ in $G$ is defined by
\begin{align*}
V(f,G)=\sup \left\{ \int_G \,f(x) \mathrm{div} \varphi(x)\, \ud x: \varphi\in C^{1}_{c}(G,\RR ^d), \norm{\varphi}_{\infty}\leq 1\right\}.
\end{align*}
The directional derivative of $f$ in $G$ in the direction $u\in \mathbb{S}^{d-1}$ is
$$V_{u}(f,G)=\sup \left\{ \int_{G}\,f(x) \sprod{\nabla \varphi(x)}{u}\, \ud x: \varphi\in C^{1}_{c}(G,\RR ^d), \norm{\varphi}_{\infty}\leq1\right\}.$$
We notice that $V(\IndFun{\Omega}, \RR ^d)$ is the perimeter $ \Per(\Omega)$ of the set $\Omega$, cf. \eqref{Perimeter_def}. Let $V_{u}(\Omega)$ stand for the quantity $V_u(\IndFun{\Omega}, \RR ^d)$.
We mention that, by \cite[Proposition 3.62]{Ambrosio_2000}, for any open $\Omega$ with Lipschitz boundary $\partial \Omega$ and finite Hausdorff measure $\sigma (\partial \Omega)$ we have 
\begin{align*}
\Per(\Omega) = \sigma (\partial \Omega).
\end{align*}

For any $\Omega \subset\RR^d$ with finite Lebesgue measure $|\Omega|$ we define the covariance function $g_\Omega$ of $\Omega$ as follows
\begin{align}\label{g_omega_defn}
g_\Omega (y)=|\Omega\cap (\Omega + y)|=\int_{\RR ^d}\,\IndFun{\Omega}(x)\,\IndFun{\Omega}(x-y) \ud x,\quad y\in \RR^d.
\end{align}
The next proposition collects all the necessary facts concerning the covariance function following the presentation of \cite{Galerne}. This also reveals the link between directional derivatives and covariance functions.
%%%%%%%%%%%%%%%%%%%%%%
\begin{proposition}{\cite[Proposition 2, Theorem 13 and Theorem 14]{Galerne}}\label{g_properties}												Let $\Omega \subset \RR ^d$ have finite measure. Then
\begin{enumerate}
\item[(i)] 
For all $y\in \RR ^d$, $0\leq g_\Omega(y)\leq g_\Omega(0)=|\Omega|$.
\item[(ii)] 
For all $y\in \RR ^d$, $ g_\Omega(y)= g_\Omega(-y)$.
%\item[(iv)] 
%$g_\Omega$ is compactly supported; for $\norm{y} \geq \mathrm{diam}(\Omega)$ we have $g_\Omega(y)=0$.
\item[(iii)] 
$g_\Omega$ is uniformly continuous in $\RR ^d$ and $\lim_{y\to \infty}g_\Omega (y)=0$.
\end{enumerate}
Moreover, if $\Omega$ is of finite perimeter $\Per(\Omega)<\infty$ then
\begin{enumerate}
\item[(iv)]
the function $g_\Omega$ is Lipschitz,
\begin{align*}
2\norm{g_\Omega}_{\mathrm{Lip}} = \sup_{u\in \mathbb{S}^{d-1}}V_u(\Omega)\leq \Per(\Omega)
\end{align*}
and
\begin{align}\label{g_lip_limit}
\lim_{r\to 0}\frac{\g(0) - \g(ru)}{|r|} = \frac{V_u(\Omega)}{2}.
\end{align}
\item[(v)] 
For all $r>0$ the limit $\lim_{r\to 0^+}\frac{\g(0)-\g(ru) }{r}$ exists, is finite and
\begin{align*}
\Per(\Omega) =  \frac{\Gamma((d+1)/2)}{\pi^{(d-1)/2}} \int_{\mathbb{S}^{d-1}}\lim_{r\to 0^+}\frac{ \g(0)- \g(ru)}{r} \sigma (\ud u).
\end{align*}
\end{enumerate}
In particular, (i) and the fact that $\g$ is Lipschitz imply that there is a constant $C=C(\Omega)>0$ such that
\begin{align}
0\leq g_\Omega (0)- g_\Omega (y) \leq C (1\wedge \norm{y}).\label{g_Omega_bound}
\end{align}
\end{proposition}
%%%%%%%%%%%%%%%%%%%%%%%%%%%%%%%%%%%%%%%%%%%%%%%%%%%%%
\subsection{Regular variation}\label{sec_RV}
A function $\ell : [x_0, +\infty) \rightarrow (0, \infty)$, for some $x_0 > 0$, is called
slowly varying at infinity if for each $\lambda > 0$
\[	
	\lim_{x \to \infty} \frac{\ell(\lambda x)}{\ell(x)} = 1.
\]
We say that $f: [x_0, +\infty) \rightarrow (0, +\infty)$ is regularly varying of index
$\alpha \in \RR$ at infinity, if $f(x) x^{-\alpha}$ is slowly varying at infinity. The set of
regularly varying functions of index $\alpha$ at infinity is denoted by $\RInf$. In particular, if $f \in \RInf$
then
\[
	\lim_{x \to \infty}
	\frac{f(\lambda x)}{f(x)}
	=\lambda^\alpha,\quad \lambda>0.
\] 
The following property, so-called \textit{Potter bounds}, of regularly varying functions will be very useful, see
\cite[Theorem 1.5.6]{bgt}. For every $C > 1$ and $\epsilon > 0$ there is $x_0=x_0(C,\epsilon)>0$ such that for all
$ x, y \geq x_0$
\begin{equation}
	\label{eq:14}
	\frac{f(x)}{f(y)}\leq C \left( (x/y)^{\alpha -\epsilon} \vee (x/y)^{\alpha +\epsilon}\right).
\end{equation}
%%%%%%%%%%%%%%%%%%%%%%%%%%
\subsection{L\'{e}vy processes}\label{sec_Levy}
Throughout the paper $\mathbf{X}$ always denotes a L\'{e}vy process, that is
a c\`{a}dl\`{a}g stochastic process with stationary and independent increments. The characteristic function of $X_t$ has the form $\EE e^{i\sprod{X_t}{\xi}} = e^{-t\psi (\xi)}$, where the characteristic exponent $\psi$ is given by \eqref{charact_expo} with the corresponding L\'{e}vy measure $\nu $, cf. \eqref{Levy_measure}.

We recall that $\mathbf{X}$ is isotropic if the measures $p_t(\ud x)$ are all radial. This is equivalent to the radiality of the L\'{e}vy measure and the characteristic exponent. For isotropic processes the characteristic exponent has the simpler form
\begin{align*}
\psi (x) = \int_{\RR^d}\left( 1- \cos \sprod{x}{y}\right)\nu (\ud x) + \eta \norm{x}^2,
\end{align*}
for some $\eta \geq 0$. 
We usually abuse notation by setting $\psi(r)$ to be equal to $\psi(x)$ for any $x \in \RR^d$
with $\norm{x} = r>0$. Since the function $\psi$ is not necessary monotone, it is more convenient to work with the non-decreasing function $\psi^*$ defined by 
\begin{equation*}
	\psi^*(u) = \sup_{s \in [0, u]} \psi(s),\quad u \geq 0.
\end{equation*}
We denote by $\psi ^-$ the generalized inverse of the function $\psi^*$, that is $\psi^-(u) = \inf \{x\geq 0:\, \psi^*(x)\geq u\}$.
By \cite[Theorem 1.5.3]{bgt}, if $\psi \in \RInf$, for some $\alpha>0$, then $\psi ^* \in \RInf$ and thus $\psi ^- \in \mathcal{R}_{1/\alpha}$, which implies that $\lim_{t\to 0}\psi^-(1/t) = \infty$.

To any L\'{e}vy process $\mathbf{X}$ we associate the function $h$ defined at \eqref{Pruitt_Function}. According to \cite[Formula (3.2)]{Pruitt}, there is some positive constant $C=C(d)$ such that for any $r>0$,
\begin{align}\label{estimate_Pruitt}
\PP \left( \norm{X_t}\geq r \right) \leq \PP \left(\sup_{0\leq s\leq t } \norm{X_s}\geq r \right)\leq C t h(r) .
\end{align}
We mention that the function $h$ is decreasing and satisfies the doubling property 
\begin{align}\label{doubling:h}
h(2x)\geq h(x)/4,\quad x>0.
\end{align}
For a symmetric L\'{e}vy process $\mathbf{X}$ the function $h$ has the simplified form
\begin{align*}
h(r)= \norm{A}r^{-2} + \int_{\RR ^d} \left( 1\wedge \norm{y}^2r^{-2} \right) \, \nu(\ud y)
\end{align*} 
and for these processes, see \cite[Corollary 1]{Grzywny1},
\begin{equation}\label{psi_star_estimate}
	\frac{1}{2} \psi^*(r^{-1}) \leq h(r) \leq 8(1 + 2d) \psi^*(r^{-1}).
\end{equation}

In the paper we also deal with L\'{e}vy processes which have finite variation on any interval $(0,t)$, for $t>0$. It holds if and only if condition \eqref{Levy_bdd_var_cond} is satisfied. 
It turns out that for such processes the quantity $\Per_{\mathbf{X}}(\Omega)$ defined at \eqref{X_perimeter} is finite.
\begin{lemma}\label{Lemma_Per_X}
Assume that $\mathbf{X}$ has finite variation. Then for any $\Omega \subset \RR ^d$ of finite measure and finite perimeter $\Per (\Omega)<\infty$ we have $\Per_{\mathbf{X}}(\Omega)<\infty$.
\end{lemma}
\begin{proof}
Using \eqref{g_Omega_bound} we can write						
\begin{align*}
\Per_{\mathbf{X}}(\Omega)&= \int_{\Omega}\int_{\Omega ^c-x}\nu (\ud y)\, \ud x = \int \int \IndFun{\Omega}(x)\IndFun{\Omega ^c}(y+x)\nu (\ud y)\ud x\\
&= \int_{\RR ^d}\left( g(0) - g(y)\right) \nu (\ud y)\leq C \int_{\RR ^d}\left( 1\wedge \norm{y}\right) \nu (\ud y).
\end{align*}
Further,
\begin{align*}
\int_{\RR ^d}\left( 1\wedge \norm{y}\right) \nu (\ud y) = \int_{\norm{y}<1}\norm{y}\nu (\ud y) 
+
\int_{\norm{y}\geq 1}\nu (\ud y),
\end{align*}
where the both integrals on the right hand side are finite due to \eqref{Levy_bdd_var_cond} and \eqref{Levy_measure},
and the proof is finished.
\end{proof}

For the detailed discussion on infinitesimal generators of semigroups related to L\'{e}vy processes we refer the reader to \cite[Section 31]{Sato} or \cite[Section 3.3]{Appl}. We recall that the heat semigroup $\{P_t\}_{t\geq 0}$ related to the L\'{e}vy process $\mathbf{X}$ is given by
\begin{align*}
P_tf(x) = \int_{\RR^d} f(x+y)p_t(\ud y),\quad f\in C_0(\RR^d),
\end{align*}
where $C_0(\RR^d)$ is the set of all continuous functions which vanish at infinity.
The generator $\mathcal{L}$ of the process $\mathbf{X}$ is a linear operator defined by
\begin{align}\label{gener}
\mathcal{L}f(x) = \lim_{t\to 0}\frac{ P_tf(x) - f(x) }{t},
\end{align}
with the domain $\mathrm{Dom}(\mathcal{L})$ which is the set of all $f$ such that the right hand side of \eqref{gener} exists.
By \cite[Theorem 31.5]{Sato}, we have $C_0^2(\RR^d)\subset \mathrm{Dom}(\mathcal{L})$ and for any $f\in C_0^2(\RR^d)$ it has the form
\begin{align*}
\mathcal{L} f(x) &= \sum_{j,k=1}^d A_{jk}\partial^2_{jk} f(x)+\sprod{\gamma}{\nabla f(x)} \\
&\quad + \int \left(f(x+z)-f(x)
-\IndFun{\norm{z}<1}\sprod{z}{\nabla f(x)} \right) \nu(\ud z),
\end{align*}
where $(A,\gamma ,\nu)$ is the triplet from \eqref{charact_expo}.  For L\'{e}vy processes with finite variation we have the following.

\begin{lemma}\label{Lemma_2}
Let $\textbf{X}^0$ be a L\'{e}vy process with finite variation and such that $\gamma _0 =0$, cf. \eqref{gamma_0}. Let $f$ be a Lipschitz function (with constant $L$) in $\RR ^d$ with $\lim _{x\to \infty}f(x) =0$. Then $f$ belongs to the domain of the generator $\mathcal{L}^0$ of the process $\textbf{X}^0$, i.e. $f\in \mathrm{Dom}(\mathcal{L}^0)$, and 
\begin{align*}
\mathcal{L}^0 f(x) = \int_{\RR^d} (f(x+y)- f(x))\nu (\ud y).
\end{align*} 
\end{lemma}
\begin{proof}
We take a function $\phi \in C_c^\infty (\RR^d)$ such that $\phi (0)=1$, $\norm{\phi}_{L^1} =1$ and $\mathrm{supp}(\phi) \subset [0,1]$. We set $\phi _{\epsilon} (x) = \epsilon ^{-d}\phi (\epsilon ^{-1} x)$. It is well known that then the function $f_{\epsilon}(x) = \phi _{\epsilon}\ast f(x)$ belongs to $C_0^\infty (\RR^d)$. Moreover, we have
$\lim _{\epsilon \to 0} \norm{f_{\epsilon} - f}_{\infty} = 0 $. Indeed, for any $\delta >0$,
\begin{align*}
|f_{\epsilon}(x) - f(x)| &\leq    \int _{\norm{y}<\delta}|\phi _{\epsilon}(y)||f(x-y)- f(x)|\, \ud y
 +  \int _{\norm{y}\geq \delta}|\phi _{\epsilon}(y)||f(x-y)- f(x)|\, \ud y \\
 &\leq L \delta  \int _{\norm{y}<\delta}|\phi _{\epsilon}(y)|\, \ud y + 2\norm{f}_{\infty} \int _{\norm{y}\geq \delta}|\phi _{\epsilon}(y)|
  \leq L \delta \norm{\phi}_{L^1} + 2\norm{f}_{\infty} \delta,
\end{align*}
for $\epsilon$ small enough. Taking $\delta $ small as well, we get the claim.

Moreover, since $\gamma_0=0$,
\begin{align*}
\mathcal{L}^0 f_{\epsilon}(x) &= \sprod{\gamma}{\nabla f_{\epsilon}(x)} + \int \left(f_{\epsilon}(x+z)-f_{\epsilon}(x)
-\IndFun{\norm{z}<1}\sprod{z}{\nabla f_{\epsilon}(x)} \right) \nu(\ud z)\\&=\sprod{\gamma_0}{\nabla f_{\epsilon}(x)}  + \int \left(f_{\epsilon}(x+z)-f_{\epsilon}(x)
 \right) \nu(\ud z)
\\
&=\int_{\RR^d} (f_{\epsilon}(x+y) - f_{\epsilon}(x))\, \nu (\ud y)
\end{align*}
and we deduce that 
\begin{align*}
\lim _{\epsilon \to 0}\mathcal{L}^0 f_{\epsilon}(x) = \int_{\RR^d} (f(x+y) - f(x))\, \nu (\ud y).
\end{align*}
Finally, since $\mathcal{L}^0$ is closed, we get that $f\in \mathrm{Dom}(\mathcal{L}^0)$ and 
\begin{align*}
\mathcal{L}^0 f(x) = \int_{\RR^d} (f(x+y) - f(x))\, \nu (\ud y)
\end{align*}
which finishes the proof.
\end{proof}
%%%%%%%%%%%%%%%%%%%%%%%%%%%%%%%%%%%%%%%%%%%%%%%%

%%%%%%%%%%%%%%%%%%%%%%%%%%%%%%%%
\section{Proofs}\label{sec_Proofs}
\subsection{Proof of Theorem \ref{Thm_d>1}}
Before we prove Theorem \ref{Thm_d>1} we establish the following auxiliary lemma. 
\begin{lemma}\label{Lemma_1}
Let $\mathbf{X}$ be a L\'{e}vy process in $\RR^d$. Then
\begin{enumerate}
	\item
		there is a constant $C=C(d) >0$ such that for any $R>0$,
			\begin{align}\label{H formula1}
			\int_0^{R} \PP (\norm{X_t} \geq x)\, \ud x &\leq C t\int_{ h^{-1}(1/t)\wedge \frac{R}{2}}^{R} h(r)\, \ud r ,\quad 				t>0.
			\end{align}
	\item 
		The related function $H(t)$ introduced in \eqref{heat_cont-H} has the following form	
			\begin{align}\label{H_formula}
				H(t) = \int _{\RR ^d}\left( g_\Omega (0) -\g (y)\right) p_t(\ud y).
			\end{align}
\end{enumerate} 
\end{lemma}
\begin{proof}
We start with the proof of (i). Using \eqref{estimate_Pruitt} we clearly get that for some $C>0$  and any $t>0$,
\begin{align*}
 \PP (\norm{X_t}\geq x)\leq C ( t  h(x) \wedge 1) .
\end{align*}
Observe that $th(x)\geq 1$ if and only if $x\leq h^{-1}(1/t)$
and thus we set $\beta = h^{-1}(1/t)$. 
For any $R>0$, we have
\begin{align}\label{est:1}
\int_0^{R} \PP (\norm{X_t} \geq x)\, \ud x \leq C\left( \int_{0}^{\beta \wedge \frac{R}{2}}\! \ud x + t \int_{\beta \wedge \frac{R}{2}}^R \! h(x)\ud x \right).
\end{align}
First we consider the case $\beta \leq R/2$, which is equivalent to $t\leq 1/h(R/2)$.
We estimate the second integral in \eqref{est:1} as follows
\begin{align*}
 t \int_{\beta }^R \! h(x)\ud x  \geq  t \int_{\beta}^{2\beta} \! h(x)\ud x\geq th(2\beta )\beta \geq \beta/4. 
\end{align*}
In the last inequality we used the doubling property \eqref{doubling:h}. We obtain that 
\begin{align*}
\int_0^{R} \PP (\norm{X_t} \geq x)\, \ud x \leq 5 C t \int_\beta ^R h(x)\ud x
\end{align*}
as desired. 
Next, assume that $R/2<\beta $. By monotonicity of $h$, we have
\begin{align*}
t \int_{R/2}^R \! h(x)\ud x\geq \frac{tR}{2} h(R)\geq \frac{tR}{2} h(2\beta )\geq R/8.
\end{align*}
This together with \eqref{est:1} imply 
\begin{align*}
\int_0^{R} \PP (\norm{X_t} \geq x)\, \ud x \leq 5 C t\int_{R/2}^R h(x)\ud x
\end{align*}
and this gives (i).

For (ii) we write
\begin{align*}
H(t) &= \int_{\RR ^d}\IndFun{\Omega}(x)\left( 1-\PP (X_t \in \Omega -x) \right) \ud x \\
	&= |\Omega| - \int_{\RR ^d}\int_{\RR ^d} \IndFun{\Omega}(x)\IndFun{\Omega}(x+y)\, \ud x \, p_t(\ud y) \\
	&= \g(0) - \int_{\RR ^d} \g (-y)\, p_t(\ud y),  
\end{align*}
and using symmetry of $\g$ (see (ii) of Proposition \ref{g_properties}) we get \eqref{H_formula}.
\end{proof}

%%%%%%%%%%%%%%%%%%%%%

%%%%%%%%%%%%%%%%%%%%%%%%%%%%%%%%%%%%%%%%%%%%%%%%%
\begin{proof}[Proof of Theorem \ref{Thm_d>1}]
We set $R = 2|\Omega|/\Per(\Omega)$ and split the integral in \eqref{H_formula} into two parts
\begin{align*}
H(t) = \int_{\norm{y} >R}(\g(0) - \g(y))\, p_t(\ud y) + \int_{\norm{y} \leq R}(\g(0) - \g(y)) \, p_t(\ud y) = \mathrm{I}_1 + \mathrm{I}_2.
\end{align*}
Using (i) of Proposition \ref{g_properties} we estimate $\mathrm{I}_1$ as follows
\begin{align*}
\mathrm{I}_1\leq |\Omega|\, \PP (||X_t||>R)\leq \frac{\Per(\Omega)}{2}\int^R_{0} \PP(||X_t||>s)\ud s.
\end{align*}
Next, by (iv) of Proposition \ref{g_properties} we obtain
\begin{align*}
\mathrm{I}_2&\leq  \frac{\Per(\Omega)}{2}\int_{||y||\leq R} ||y||\,p_t(\ud y)=
 \frac{\Per(\Omega)}{2}\int_{||y||\leq R} \int^{||y||}_0\ud s\, p_t(\ud y)\\& =  \frac{\Per(\Omega)}{2}\int^R_0\int_{s<||y||\leq R} ||y||\,p_t(\ud y)\ud s\leq  \frac{\Per(\Omega)}{2}\int^R_0\PP(||X_t||>s) \ud s.
\end{align*}
These estimates together imply 
\begin{align*}
H(t)\leq \Per(\Omega)\int^R_0\PP(||X_t||>s) \ud s.
\end{align*}
Hence, applying (i) of Lemma 3 we deduce the result.
\end{proof}
In the following Proposition \ref{H_lower_bound}, we provide a lower bound for the heat content related to an isotropic L\'{e}vy process with the characteristic exponent satisfying some scaling condition. We start with a useful lemma.
\begin{lemma}\label{lem:14}
Let $\textbf{X}$ be an isotropic L\'{e}vy process with the radial characteristic exponent $\psi$. Suppose that there is a constant $C>0$ such that for some $\alpha \in (0,2)$, 
\begin{align}\label{WUSC_condition1}
C^{-1}\psi(x)\leq \psi(y) \leq  C \left(\frac{y}{x}\right)^{\alpha} \psi(x),\quad 1<x<y.
\end{align}
Then there exists $c>0$ such that
\begin{align*}
\PP(||X_t||>r) &\geq c\, (1-e^{t \,h(r)}), \quad t,\,r<1.
\end{align*}
\end{lemma}
\begin{proof}We observe that the left hand side inequality in \eqref{WUSC_condition1} implies that
\begin{align}\label{psi_psi*}
\psi (x)\geq C^{-1}\psi^*(x),\quad x>1.
\end{align}
Thus, proceeding exactly in the same fashion as in the proof of \cite[Lemma 14]{BGR}, we obtain
\begin{align*}
\PP(||X_t||>r) &\geq C_1 (1-e^{t \,\psi^*(1/r)}), \quad t,\,r<1.
\end{align*} 
Finally, inequality \eqref{psi_star_estimate} yields
\begin{align*}
\psi^*(r)\geq ch(1/r),\quad r>0,
\end{align*} 
and the proof is finished. 
\end{proof}
\begin{proposition}\label{H_lower_bound}
Let $\textbf{X}$ be an isotropic L\'{e}vy process in $\RR^d$ with the radial characteristic exponent $\psi$ which satisfies condition \eqref{WUSC_condition1}.
Assume also that the related function $h$ (see \eqref{Pruitt_Function}) is not Lebesgue integrable around zero. Then, for any open $\Omega \subset \RR^d$ with finite measure and of finite perimeter, there exists $C>0$ which does not depend on the set $\Omega$ such that, for $t$ small enough,
\begin{align*}
H(t) &\geq C \,t\,\Per(\Omega)\, \int_{h^{-1}(1/t)}^{R} h(r)\, \ud r ,
\end{align*}
where $R=2|\Omega|/\Per(\Omega)$.
\end{proposition}
\begin{proof}
We first consider the case $d\geq2$. Since $h$ is not integrable around $0$, it is unbounded and so does $\psi$ due to inequality \eqref{psi_star_estimate}. 
Therefore, $\textbf{X}$ is not a compound Poisson process. Hence, by \cite[(4.6)]{Zabczyk}, all the transition probabilities $p_t(\ud x)$ are absolutely continuous with respect to the Lebesgue measure.
Since $p_t$ are radial, we have $p_t(x) = p_t(\norm{x}e_d)$ with $e_d=(0,\ldots ,0,1)$ and as a result by polar coordinates we get that, for any $u_1,u_2\in [0 , +\infty ]$,
\begin{align}\label{polar_formula}
\PP \left( u_1< \norm{X_t} < u_2\right) = \int_{\RR^d}\IndFun{\{(u_1,u_2)\}}(\norm{w})p_t(w)\ud w = \sigma (\mathbb{S}^{d-1})\int_{u_1}^{u_2}r^{d-1}p_t(re_d)\, \ud r.
\end{align}
Applying \eqref{polar_formula} in \eqref{H_formula} we obtain that for any $\delta >0$,
\begin{align*}
H(t)\geq \int_0^\delta r^{d-1} p_t(r e_d) \int_{\mathbb{S}^{d-1}} \left( g_\Omega (0) - g_\Omega (ru)\right)\,
 \sigma (\ud u)\, \ud r 
 = \int_0^\delta r^{d} p_t(r e_d)  \, \mathcal{M}_{\Omega}(r)\, \ud r,
\end{align*} 
where 
\begin{align*}
\mathcal{M}_{\Omega}(r) = \int_{\mathbb{S}^{d-1}} \frac{g_\Omega (0) - g_\Omega (ru)}{r}\,
\sigma (\ud u).
\end{align*}
Using (v) of Proposition \ref{g_properties} and applying Fatou's lemma, we get that $\mathcal{M}_{\Omega}(r) \geq C\Per(\Omega)$, for some positive constant $C=C(d)$ and for $r$ small enough.
Hence, for $\delta$ small enough,
\begin{align*}
H(t)\geq C\Per(\Omega)\int_0^\delta r^{d} p_t(r e_d) \, \ud r &= C\Per(\Omega)\int_0^\delta \int_0^r \ud u\,  r^{d-1} p_t(r e_d) \, \ud r \\
&= \frac{C\Per(\Omega)}{\sigma (\mathbb{S}^{d-1})}\int_0^\delta \int_{u<||y||<\delta}p_t(y)\ud y\, \ud u \\
&= \frac{C\Per(\Omega)}{\sigma (\mathbb{S}^{d-1})}\int_0^\delta \PP(u<\norm{X_t}<\delta ) \, \ud u \\
&=  \frac{C\Per(\Omega)}{\sigma (\mathbb{S}^{d-1})}\left(\int_0^\delta  \PP(\norm{X_t} >u ) \, \ud u - \delta\PP(\norm{X_t} > \delta)\right),
\end{align*}
where in the second equality we used \eqref{polar_formula}.
By Lemma \ref{lem:14}, there is a constant $C_1=C_1(d)>0$ such that, for $u$ small enough,
\begin{align*}
\PP(\norm{X_t} >u )\geq C_1 t h(u).
\end{align*}
This and \eqref{estimate_Pruitt} imply that, for $\delta$ small enough,
\begin{align}\label{int_234}
H(t)\geq C_2\, t\,\Per(\Omega)\,  \left( \int_{h^{-1}(1/t)}^\delta  h(u)\, \ud u - C_3 \delta  h(\delta) \right).
\end{align} 
Since $\psi$ is continuous and $\psi(0)=0$ we have by \eqref{psi_star_estimate}, 
\begin{align*}
\lim_{t\to 0}h^{-1}(1/t) = 0.
\end{align*}
Further, since $h$ is not integrable around zero, the integral on the right hand side of \eqref{int_234} tends to infinity for any $\delta >0$, as $t$ goes to zero.
This implies that, for $t$ small enough,
\begin{align*}
\int_{h^{-1}(1/t)}^{\delta}  h(u)\, \ud u - C_3 \delta  h(\delta)&= \int_{h^{-1}(1/t)}^{R}  h(u)\, \ud u- \int_{\delta}^{R}  h(u)\, \ud u- C_3 \delta  h(\delta)\\
&= \int_{h^{-1}(1/t)}^{R}  h(u)\, \ud u \left( 1-\frac{\int_{\delta}^{R}  h(u)\, \ud u+ C_3 \delta  h(\delta)}{ \int_{h^{-1}(1/t)}^{R}  h(u)\, \ud u} \right)\\
& \geq \frac{1}{2} \int_{h^{-1}(1/t)}^{R}  h(u)\, \ud u .
\end{align*}
Using this and \eqref{int_234} we obtain that there is some $C_4>0$ which does not depend on $\Omega$ such that, for $t$ small enough, 
\begin{align*}
H(t)\geq C_4 \, t\,\Per(\Omega)\,  \int_{h^{-1}(1/t)}^{R}  h(u)\, \ud u,
\end{align*}
and the proof is finished for $d\geq 2$.

At last, in the case $d=1$ we use \eqref{H_formula}, 
\begin{align*}
H(t)\geq \int_0^\delta  \frac{g_\Omega (0) - g_\Omega (x)}{x}xp_t(\ud x),
\end{align*}
and application of (v) of Proposition \ref{g_properties} with $d=1$ gives that, for $0<x$ small enough,
\begin{align*}
\frac{g_\Omega (0) - g_\Omega (x)}{x}\geq  C\Per(\Omega).
\end{align*}
Thus, for $\delta$ small enough, by symmetry of $\textbf{X}$,
\begin{align*}H(t)&\geq C\Per(\Omega) \int_0^\delta xp_t(\ud x) =  C\Per(\Omega) \int _0^\delta \int_0^{x}\ud u\, p_t(\ud x)\\
& = C\Per(\Omega) \int _0^\delta \int_{u}^{\delta} p_t(\ud x)\, \ud u\\
& = \frac{C}{2}\Per(\Omega) \int _0^\delta \PP(u<|X_t|<\delta)\, \ud u.
\end{align*}
The result is concluded by the same reasoning as for $d\geq2$. 
\end{proof}

%%%%%%%%%%%%%%%%%%%%
\subsection{Proof of Theorem \ref{Thm_alpha>1}}
We start with the following auxiliary lemma.
\begin{lemma}\label{lemma444}
Let $\mathbf{X}$ be an isotropic L\'{e}vy process in $\RR^d$ with the radial transition probability $p_t(\ud x)$. Assume that its characteristic exponent $\psi \in \RInf$ with $\alpha \in (1,2]$. Then $p_t(\ud x)=p_t(x)\ud x$ and
\begin{align}\label{limit_stable}
\lim_{t\to 0} \frac{p_t\left( \frac{s}{\psiI}e_d \right)}{(\psiI )^d} = p_1^{(\alpha)}(se_d),
\end{align}
where $p_t^{(\alpha)}(x)$ is the transition density of the isotropic $\alpha$-stable process in $\RR^d$ when $1<\alpha <2$ and $p_t^{(2)}(x)$ is the transition density of the Brownian motion in $\RR^d$. 
\end{lemma}
\begin{proof}
Since $\psi \in \RInf$, $\alpha \in (1,2]$, we have
\begin{align*}
\lim_{r\to \infty}\frac{\psi (r)}{\log (1+r)}=\infty,
\end{align*}
and this implies that $p_t(\ud ) = p_t(x)\ud x$ with the density $p_t\in L_1(\RR^d)\cap C_0(\RR^d)$, see e.g. \cite[Theorem 1]{Knopova_Schilling}.

By the Fourier inversion formula, see \cite[Section 3.3]{Appl}, 
\begin{align}\label{integral_1}
\frac{p_t\left( \frac{s}{\psiI}e_d \right)}{(\psiI )^d} =
\frac{1}{(2\pi)^d} \int_{\RR ^d}\cos \sprod{se_d}{\xi} e^{-t\psi \left( \psiI \xi\right)}\ud \xi .
\end{align}
Since $\psi$ is continuous, $\psi(\psi^{-}(1/t))=1/t$. Hence
\begin{align*}
\frac{\psi \left( \psiI \xi\right)}{1/t}= \frac{\psi \left( \psiI \xi\right)}{\psi \left( \psiI \right)}\sim \norm{\xi}^\alpha,\quad t\to 0, 
\end{align*}
and this leads to
\begin{align*}
\lim_{t\to 0}e^{-t\psi \left( \psiI \xi\right)} = e^{-\norm{\xi}^\alpha}.
\end{align*}
Therefore, to finish the proof we apply the Dominated convergence theorem. 
We split the integral in \eqref{integral_1} into two parts. According to the Potter bounds \eqref{eq:14} there is $r_0>0$ such that, for $t$ small enough and $\norm{\xi}\geq r_0$, 
\begin{align*}
t\psi \left( \psiI \xi\right)=\frac{\psi \left( \psiI \xi\right)}{\psi \left( \psiI \right)}\geq \frac{1}{2} \norm{\xi}^{\alpha/2}.
\end{align*}
This implies that $e^{-t\psi \left( \psiI \xi\right)}\leq e^{-\norm{\xi}^{\alpha/2}/2}$, for $\norm{\xi}\geq r_0$ and $t$ small enough.
For $\norm{\xi}< r_0$ we bound $e^{-t\psi \left( \psiI \xi\right)}$ by one. 
The Dominated convergence theorem followed by the Fourier inversion formula proves \eqref{limit_stable}.
\end{proof}
%%%%%%%%%%%%%%%%%%%%%%%%
\begin{proof}[Proof of Theorem \ref{Thm_alpha>1}.]
By \eqref{H_formula},
\begin{align*}
 H(t) =  \int_{\RR ^d} p_t(x)\left( g_\Omega (0) - g_\Omega (x)\right) \ud x.
\end{align*}
Since $g_\Omega (x) \leq g_\Omega (0) = |\Omega|$, $x\in \RR ^d$ (see Proposition \ref{g_properties} (i)), for any given $\delta >0$, 
we can split the integral into two parts
\begin{align}\label{split_delta}
  H(t) &=  \int_{\norm{x}\leq \delta} p_t(x)\left( g_\Omega (0) - g_\Omega (x)\right) \ud x + 
  \int_{\norm{x}> \delta} p_t(x) \left( g_\Omega (0) - g_\Omega (x)\right) \ud x \\
  &= \mathrm{I}_1+\mathrm{I}_2.\nonumber
\end{align}
We estimate $\mathrm{I}_2$ using \eqref{estimate_Pruitt},
\begin{align*}
 \int_{\norm{x}> \delta} p_t(x) \left( g_\Omega (0) - g_\Omega (x)\right) \ud x \leq |\Omega| \PP (\norm{X}_t >\delta) = O(t).
\end{align*}
Since $\psi \in \RInf$, $1<\alpha \leq 2$, \cite[Theorem 1.5.12]{bgt} yields $\psi ^{-}\in \mathcal{R}_{1/\alpha}$ and thus $\psiI \,\mathrm{I}_2\to 0$ as $t$ tends to zero.
We are left to study the integral $\mathrm{I}_1$.
Recall that the radiality of $p_t$ implies that $p_t(x) = p_t(re_d)$, where $\norm{x} = r$ and $e_d=(0,\ldots,0,1)$.
Changing variables into polar coordinates we obtain
\begin{align*}
\psiI \, \mathrm{I}_1 = \psiI \int_0^\delta r^{d-1} p_t(re_d) \int_{\mathbb{S}^{d-1}} \left( g_\Omega (0) - g_\Omega (ru)\right)\,
 \sigma (\ud u)\, \ud r.
\end{align*}
Making substitution $r=s/\psiI$ we get
\begin{align*}
\psiI \, \mathrm{I}_1 =  \int_0^{\delta \psiI }\!\! s^{d}\, \frac{p_t\left( \frac{s}{\psiI}e_d \right)}{(\psiI )^d}\, \mathcal{M}_{\Omega}(t,s)\, \ud s, 
\end{align*}
where 
\begin{align*}
\mathcal{M}_{\Omega}(t,s) = \int_{\mathbb{S}^{d-1}} \frac{g_\Omega (0) - g_\Omega \left(\frac{s}{\psiI}u\right)}{s/\psiI}\,
\sigma (\ud u).
\end{align*}
We claim that for any fixed $M>0$,
\begin{align}\label{claim111}
\lim_{t\to 0} \int_0^{M }\!\! s^{d}\, \frac{p_t\left( \frac{s}{\psiI}e_d \right) }{(\psiI )^d}\, \mathcal{M}_{\Omega}(t,s)\, \ud s
=
\frac{\pi^{(d-1)/2}}{\Gamma\left((d+1)/2\right)}\Per (\Omega)
\int_0^M s^d p_1^{(\alpha)}(se_d)\ud s.
\end{align}
To show the claim we use the Dominated convergence theorem. By Proposition \ref{g_properties} (iv-v),
\begin{align*}
0\leq \mathcal{M}_{\Omega}(t,s)\leq \frac{1}{2} \Per(\Omega)\, \sigma (\mathbb{S}^{d-1}) 
\end{align*}
and, for any $s>0$,
\begin{align*}
\lim_{t\to 0}\mathcal{M}_{\Omega}(t,s) = \frac{\pi^{(d-1)/2}}{\Gamma\left((d+1)/2\right)}\Per (\Omega).
\end{align*}
Next, by \cite[Formula (23)]{BGR}, for $s\leq M$,					 
\begin{align*}
\frac{p_t\left( \frac{s}{\psiI}e_d \right)}{(\psiI )^d} \leq \frac{p_t(0)}{(\psiI )^d}
\leq C(M)
\end{align*}
and, by Lemma \ref{lemma444},				
\begin{align*}
\lim_{t\to 0} \frac{p_t\left( \frac{s}{\psiI}e_d \right)}{(\psiI )^d} = p_1^{(\alpha)}(se_d).
\end{align*}
Hence the Dominated convergence theorem implies \eqref{claim111}.

Further, we have
\begin{multline*}
\int_M^{\delta \psiI }\!\! s^{d}\, \frac{p_t\left( \frac{s}{\psiI}e_d \right)}{(\psiI )^d}\, \ud s
=
\psiI \int_{M/\psiI}^{\delta  }s^d p_t(s e_d)\, \ud s\\
=
\psiI \int_{M/\psiI}^{\delta  }\int_0^s \ud u\, s^{d-1} p_t(s e_d)\, \ud s
=
\psiI \int_0^\delta \int_{(M/\psiI )\vee u}^{\delta  }\!\!\!\!  s^{d-1} p_t(s e_d)\, \ud s\, \ud u\\
\leq 
\psiI \int_0^\delta \PP \left( \norm{X_t}> (M/\psiI )\vee u\right)\ud u \\
\leq 
M \PP \left( \norm{X_t}> M/\psiI \right) 
+ \psiI \int_{M/\psiI}^\delta \PP \left( \norm{X_t}> u \right)\ud u .
\end{multline*}
Now, we notice that combinig  \eqref{estimate_Pruitt} and \eqref{psi_star_estimate}, we obtain
$\PP(\norm{X_t} > r) \leq C\, t\psi^{*}(1/r) $.
Thus, using Potter bounds with $\epsilon <\alpha -1$, we estimate, for $t$ small enough, the first term as follows
\begin{align*}
M \PP \left( \norm{X_t}> M/\psiI \right) &\leq Mt\psi^*\left( \psiI /M \right)\\ 
&\leq C_1 M \frac{\psi^*\left( \psiI /M \right)}{\psi^*\left( \psiI \right)}
\leq C_2 M^{1-\alpha +\epsilon}.
\end{align*}
We proceed similarly with the second term. Applying Karamata's theorem \cite[Proposition 1.5.8]{bgt} and Potter bounds we obtain that for $t$ small enough
\begin{multline*}
\psiI \int_{M/\psiI}^\delta \PP \left( \norm{X_t}> u \right)\ud u \leq t\psiI  \int_{M/\psiI}^\delta \psi^*(u^{-1})\, \ud u \\ 
\leq C_3 M(\alpha +1)^{-1} t\, \psi^*\left( \psiI /M \right)\\
\leq 
C_4 M(\alpha +1)^{-1} \frac{\psi^*\left( \psiI /M \right)}{\psi^*\left( \psiI \right)}\leq C_5 (\alpha +1)^{-1}M^{1-\alpha +\epsilon}.
\end{multline*}
 Finally, letting $M$ to infinity we obtain
\begin{align*}
\lim_{t\to 0}\psiI \,\mathrm{I}_1 = \frac{\pi^{(d-1)/2}}{\Gamma\left((d+1)/2\right)}\Per (\Omega)\int_{0}^\infty s^d p_1^{(\alpha)}(se_d)\ud s .
\end{align*}
It is known that, see e.g. \cite[Lemma 4.1]{Valverde2},
\begin{align*}
\int_{0}^\infty s^d p_1^{(\alpha)}(se_d)\ud s  = \pi^{-(d+1)/2}\Gamma ((d+1)/2)\Gamma(1-1/\alpha)
\end{align*}
and we conclude the result.
\end{proof}

%%%%%%%%%%%%%%%%%%%%%%%%%%%%%%%%%%%%

\subsection{Proof of Theorem \ref{Thm_X_bdd_variation}}
\begin{proof}[Proof of Theorem \ref{Thm_X_bdd_variation}]
We consider two cases: the first is $\gamma_0 =0$. Then $X_t = X_t^0$, where $\mathbf{X}^0$ is as in Lemma \ref{Lemma_2} and we have
\begin{align}\label{eq222}
t^{-1}H(t) = \int _{\Omega}\frac{1 - \PP (X_t +x \in \Omega)}{t}\, \ud x = t^{-1}(\g(0) - P_t\g (0)),
\end{align}
which converges to $ -\mathcal{L}\g(0) = \Per _{\mathbf{X}}(\Omega)$ according to Subsection \ref{sec_Levy} and Lemma \ref{Lemma_2}.

In the case $\gamma _0 \neq 0$,
we write $X_t = X_t^0 +t \gamma_0 $, where $\mathbf{X}^0$ is again as in Lemma \ref{Lemma_2}, and thus
\begin{align*}
t^{-1}H(t) &= \int _{\Omega}\frac{1 - \PP (X^0_t + t \gamma _0  +x \in \Omega)}{t}\, \ud x \\
	&= 	 \int _{\Omega}\frac{1 - \PP (X^0_t +x \in \Omega)}{t}\, \ud x +
	  \int _{\Omega}\frac{\PP (X^0_t +x \in \Omega) - \PP (X^0_t + t \gamma _0  +x \in \Omega)}{t}\, \ud x .
\end{align*}
By \eqref{eq222} we obtain
\begin{align}\label{eq222a}
\lim_{t\to 0} \int _{\Omega}\frac{1 - \PP (X^0_t +x \in \Omega)}{t}\, \ud x = \Per _{\mathbf{X}}(\Omega).
\end{align}
We denote by $p^0_t(\ud x)$
and $h^0$ the transition probabilities and the function introduced in \eqref{Pruitt_Function}, respectively,
corresponding to the process $\mathbf{X}^0$. 
For the second integral we proceed as follows
\begin{multline*}
	\int _{\Omega}\frac{\PP (X^0_t +x \in \Omega) - \PP (X^0_t + t \gamma _0  +x \in \Omega)}{t}\, \ud x \\
	=	\int_{\RR^d} \IndFun{\Omega}(x)t^{-1}\int_{\RR^d}\left( \IndFun{\Omega}(y+x) - \IndFun{\Omega}(y+x+t \gamma_0 ) \right)\, p_t^0(\ud y)\, \ud x\\
	= \int_{\RR^d}\frac{\g(y) - \g(y + t \gamma _0  )}{t}\, p_t^0(\ud y).
\end{multline*}
We take $\epsilon >0$. Using (iv) of Proposition \ref{g_properties} and \eqref{estimate_Pruitt} we write
\begin{align}\label{eq111}
\begin{aligned}
\Big| \int_{\norm{y}>\epsilon t} \frac{\g(y) - \g(y + t \gamma _0  )}{t}\, p_t^0(\ud y) \Big| &\leq \norm{\gamma_0}\,\PP\big(\norm{X^0_t}>\epsilon t\big)\leq C \norm{\gamma_0}\,t h^0(\epsilon t)\\
=  C \norm{\gamma_0}t \int _{\RR^d} \left( 1\wedge \frac{\norm{y}^2}{(\epsilon t)^2} \right)\,\nu (\ud y) 
&=  C \norm{\gamma_0}t \int _{\RR^d} \left( 1\wedge \frac{\norm{y}}{\epsilon t} \right)^2 \,\nu (\ud y) \\
\leq  C \norm{\gamma_0}t \int _{\RR^d} \left( 1\wedge \frac{\norm{y}}{\epsilon t} \right) \,\nu (\ud y) 
 &= C \norm{\gamma_0} \int _{\RR^d} \left( t\wedge \frac{\norm{y}}{\epsilon } \right) \,\nu (\ud y) .
 \end{aligned}
\end{align}
By the Lebesgue dominated convergence theorem the last quantity tends to zero as $t$ goes to zero. For the other part of the integral we have
\begin{multline*}
\int_{\norm{y}\leq \epsilon t} \frac{\g(y) - \g(y + t \gamma _0  )}{t}\, p_t^0(\ud y) 
   =	\int_{\norm{y}\leq \epsilon t} \frac{\g(y) - \g(0 )}{t}\, p_t^0(\ud y)	\\
	 + \int_{\norm{y}\leq \epsilon t} \frac{\g(0) - \g(t \gamma _0  )}{t}\, p_t^0(\ud y) +
	 	\int_{\norm{y}\leq \epsilon t} \frac{\g(\gamma_0 t) - \g(y + t \gamma _0  )}{t}\, p_t^0(\ud y)
	 	= I_1 + I_2 +I_3.
\end{multline*}
Handling with $I_1$ is easy in front of condition (iv) of Proposition \ref{g_properties}. Indeed,
\begin{align*}
|I_1| \leq \int_{\norm{y}\leq \epsilon t} \frac{|\g(y) - \g(0)|}{\norm{y}}\cdot \frac{\norm{y}}{t}\, p_t^0(\ud y)\leq 
L\epsilon  \int_{\norm{y}\leq \epsilon t}  p_t^0(\ud y)
\leq L\epsilon .
\end{align*}
Similarly we estimate $I_3$. The integral $I_2$ equals
\begin{align*}
I_2 = \frac{\g(0) - \g\left(( \norm{\gamma _0}t ) \frac{  \gamma _0}{\norm{\gamma _0}}\right)}{\norm{\gamma _0 } t}\norm{\gamma _0}
\int_{\norm{y}\leq \epsilon t} p_t^0(\ud y).
\end{align*}
Using \eqref{g_lip_limit} we obtain
\begin{align}\label{eq333}
\lim_{t\to 0}\frac{\g(0) - \g\left((\norm{\gamma _0} t) \frac{\gamma _0}{\norm{\gamma _0}} \right)}{\norm{\gamma _0}t} = 
\frac{V_{\frac{\gamma _0}{\norm{\gamma _0}}}(\Omega)}{2}.
\end{align}
Moreover, we claim that 
\begin{align*}
\lim_{t\to 0}\int_{\norm{y}\leq \epsilon t} p_t^0(\ud y) =1.
\end{align*}
Indeed, we have
\begin{align*}
\int_{\norm{y}\leq \epsilon t} p_t^0(\ud y)  = 1- \PP\left( \norm{X_t^0}>\epsilon t \right).
\end{align*}
Proceeding in the same fashion as in \eqref{eq111} we show that $\PP\left( \norm{X_t^0}>\epsilon t \right)$ tends to zero as $t$ goes to zero, which gives the claim.
Finally, equations \eqref{eq222a} and \eqref{eq333} imply the result.
\end{proof}
%%%%%%%%%%%%%%%%%%%%

\bibliographystyle{plain}

\end{document}